\documentclass[11pt]{amsart}

\usepackage{amsmath,amssymb,amsthm}
\usepackage{geometry}
\usepackage{xcolor}
\usepackage{hyperref}
\usepackage[numbers,sort&compress]{natbib}

\newtheorem{theorem}{Theorem}[section]
\newtheorem{lemma}[theorem]{Lemma}

\theoremstyle{definition}
\newtheorem{definition}[theorem]{Definition}
\newtheorem{example}[theorem]{Example}
\newtheorem{conjecture}[theorem]{Conjecture}
\theoremstyle{remark}

\numberwithin{equation}{section}

\newenvironment{ack}{\section*{Acknowledgments}}{}

\title{On the mathematics table problem}

\author{Xiao-Song Yang}
\thanks{Corresponding author: XIAO-SONG YANG}
\address{School of Mathematics and Statistics,
Huazhong University of Science and Technology,
1037 Luoyu Road, 430074 Wuhan, P.R. China}
\address{Hubei Key Laboratory of Engineering Modeling
and Scientific Computing, Huazhong University of Science
and Technology, Wuhan 430074, P.R. China}
\email{yangxs@hust.edu.cn}

\author{Xuan Zhou}
\address{School of Mathematics and Statistics,
Huazhong University of Science and Technology,
1037 Luoyu Road, 430074 Wuhan, P.R. China}
\email{xuanzhou1037@hust.edu.cn}



\begin{document}

\begin{abstract}
In this paper we study the mathematical table problem from a geometric-topological 
point of view. We prove a zero-existence theorem on a cylinder, which gives a new
proof of Fenn's square-table theorem under Fenn’s boundary conditions, and 
establish a variant under different boundary conditions. We also prove that every square 
table admits a horizontal placement on saddle surfaces. Finally, we show that 
almost every level set of a smooth Fenn graph contains a rectangle similar to any 
prescribed rectangle and an orientation-preserving similar copy of every prescribed 
cyclic quadrilateral.
\end{abstract}

\maketitle

\section{Introduction}\label{sec:intro}

Let $V=[v_{1},v_{2},v_{3},v_{4}]\subset\mathbb{R}^{2}$ be the set of
vertices of a square, with $\overline{v_{1}v_{3}}$ and
$\overline{v_{2}v_{4}}$ being the diagonals. Let $G(f)$ be the
graph defined by a continuous function $f:\mathbb{R}^{2}\to\mathbb{R}$,
that is $G(f)=\{(x,f(x))\in\mathbb{R}^{3}: x\in\mathbb{R}^{2}\}$.
The well-known table problem can be mathematically formulated as
follows.

Given a planar configuration
$V=[v_{1},v_{2},v_{3},v_{4}]\subset\mathbb{R}^{2}$, one may ask
whether there exists a planar rigid motion
$T:\mathbb{R}^{2}\to\mathbb{R}^{2}$ such that
\[
  f(Tv_{1})=f(Tv_{2})=f(Tv_{3})=f(Tv_{4}).
\]
Equivalently, the four corresponding points
\[
  (Tv_i,f(Tv_i))\in G(f),\qquad i=1,2,3,4,
\]
lie in a horizontal plane.

A related question that remains open is whether such a fixed-size
horizontal placement always exists when $V$ is the set of vertices
of a rectangle.

Inspired by Dyson's theorem on real-valued continuous maps on the
$2$-dimensional unit sphere $S^{2}$, R.~Fenn proved the following
theorem, which has an intuitive explanation for horizontally
balancing a square table on an uneven floor or hills.

\begin{theorem}[Fenn's theorem]\label{thm:fenn}
  Let $D\subset\mathbb{R}^{2}$ be a convex compact set. Let $f:\mathbb{R}^{2}\to\mathbb{R}$ be a continuous map satisfying
the following conditions:
  \begin{enumerate}
    \item[(i)] $f(x)=0$ if $x\in\mathbb{R}^{2}-D$;
    \item[(ii)] $f(x)\ge 0$ if $x\in D$.
  \end{enumerate}
  Then for any $a>0$, there is a square in the plane
  $\mathbb{R}^{2}$ centered at a point $p\in D$, with side length
  $a$, such that $f$ takes the same value on the four vertices of
  this square.
\end{theorem}

Fenn's theorem depends on the specific boundary conditions on the
continuous map.

As in \cite{Fenn70}, let $D\subset\mathbb{R}^{2}$ be a compact
convex domain with the origin in its interior. Denote by $a_{1},a_{2},a_{3}$ and $a_{4}$ the vertices
of a square in the plane with center $p$ and fixed side length $d$.
Let $\Pi(p,\alpha)$ be the plane determined by the points
$(a_{1},f(a_{1}))$, $(a_{2},f(a_{2}))$ and $(a_{3},f(a_{3}))$,
with $\alpha$ being the angle between the $x_{1}$-axis and the
vector $\overline{a_{1}a_{2}}$.  Denote by $L(a_{4})$ the line
parallel to the $x_{3}$-axis passing through $a_{4}$. Let $q(p,\alpha)$ denote the $x_{3}$-coordinate of the point 
$\Pi(p,\alpha)\cap L(a_{4})$.  Then the function
$\phi:D\times S^{1}\to\mathbb{R}$ defined by
$\phi(p,\alpha)=q(p,\alpha)-f(a_{4})$ is clearly continuous because the
fourth vertex $a_{4}$ depends continuously on the center point $p$
and the angle $\alpha$.

The function $\phi$ has the following properties.
\begin{itemize}
  \item[(i)] If $\phi(x,\alpha)=0$, then
    $\phi(x,\alpha+\frac{\pi}{2})=0$.
  \item[(ii)] If $\phi(x,\alpha)\neq0$, then
    $\phi(x,\alpha)\cdot\phi(x,\alpha+\frac{\pi}{2})<0$.
\end{itemize}

For $(x,\alpha)\in D\times S^{1}$, let $n(x,\alpha)$ be the unit
vector normal to the plane $\Pi(x,\alpha)$ with the
$x_{3}$-component being positive.  Let $v(x,\alpha)=P(n(x,\alpha))$, where $P:\mathbb{R}^{3}\to\mathbb{R}^{2}$ is the horizontal projection $P(x_{1},x_{2},x_{3})=(x_{1},x_{2})$.  We obtain
a map $\Phi:D\times S^{1}\to\mathbb{R}^{3}$ as follows
\begin{equation}\label{eq:Phi}
  \Phi(x,\alpha)=
  \begin{pmatrix}
    \phi(x,\alpha)\\[2pt]
    \psi(x,\alpha)
  \end{pmatrix},
\end{equation}
with $\psi(x,\alpha)=v(x,\alpha)+v(x,\alpha+\frac{\pi}{2})
  +v(x,\alpha+\pi)+v(x,\alpha+\frac{3\pi}{2})$.

Note that $\psi(x,\alpha)=\psi(x,\alpha+\frac{\pi}{2})$.

To prove Fenn's theorem, it is clearly enough to prove the
existence of zero points of the map $\Phi$.  For points
$(x,\alpha)\in\phi^{-1}(0)$, we have $\psi(x,\alpha)=4v(x,\alpha)$.

Unlike \cite{Fenn70}, where the arguments are based on homological
techniques in algebraic topology, we present an elementary proof in
terms of differential topology for this elegant theorem. In addition, 
we discuss related problems for rectangular and cyclic
configurations, especially on regular levels of smooth Fenn graphs.

This paper is organized as follows. In Section~\ref{sec:basic} we establish the
zero-existence theorem. In Section~\ref{sec:table} we apply it to give a new proof
of Fenn's theorem and give a new theorem under different boundary conditions. 
In Section~\ref{sec:surfaces} we discuss square tables on saddle surfaces.
Finally, Section~\ref{sec:cyclic} establishes two theorems on rectangular and cyclic
configurations on regular level sets of smooth Fenn graphs and propese a conjecture 
for cyclic quadrilateral.

\section{A Basic Theorem}\label{sec:basic}

In this section we mainly establish a general result for later
arguments. Let $D\subset\mathbb{R}^{2}$ be a compact domain 
 with $C^{1}$ boundary. We first record the following elementary fact.

\begin{lemma}\label{lem:basic}
  Let $F:D\times[a,b]\to\mathbb{R}^{2}$ be a smooth map satisfying:
  \begin{enumerate}
    \item[(1)]
      $F^{-1}(0)\cap(\partial D\times[a,b])=\varnothing$;
    \item[(2)]
      $0$ is a regular value of $F$, and also a regular value of
      $F_{a}=F|_{D\times\{a\}}$ and $F_{b}=F|_{D\times\{b\}}$;
    \item[(3)]
      $\deg\bigl(\frac{F_{a}}{\|F_{a}\|}:\partial D\to S^{1}\bigr)$
      is odd.
  \end{enumerate}
  Then $F^{-1}(0)$ contains a connected component meeting both
  $D\times\{a\}$ and $D\times\{b\}$.
\end{lemma}

\begin{proof}
  By the Poincar\'e--Hopf index theorem, condition~(3) implies that
  $F_a$ has an odd number of zero points in $D$.

  Since $0$ is a regular value of $F$, the zero set $F^{-1}(0)$ is a
  compact one-dimensional manifold, whose components are circles or
  intervals. Moreover, by condition (1), the zero set
	does not meet $\partial D\times[a,b]$. Therefore the endpoints of these
	segments, if any, can only lie in $D\times\{a\}$ or in $D\times\{b\}$.
	
	Suppose that no segment connects $D\times\{a\}$ and $D\times\{b\}$.
	Then every segment with one endpoint in $D\times\{a\}$ must have its
	other endpoint also in $D\times\{a\}$. It follows that the number of
	zero points of $F$ on $D\times\{a\}$ must be even. This contradicts the
	oddness proved above. Therefore $F^{-1}(0)$ has a component connecting
	$D\times\{a\}$ and $D\times\{b\}$.
\end{proof}

We now prove the main result of this section, which will be applied to
the table problem in the next section.

\begin{theorem}\label{thm:main}
Let $D\subset\mathbb{R}^{2}$ be a compact convex domain with $C^{1}$
boundary, and let
\[
  F:D\times[a,b]\to\mathbb{R}^{2},
  \qquad
  G:D\times[a,b]\to\mathbb{R}
\]
be continuous maps satisfying:
\begin{enumerate}
  \item[(1)] $F(x,s)\cdot\nu_D(x)>0$ for all
    $x\in\partial D$ and $s\in[a,b]$, where $\nu_D(x)$ denotes
    the outward unit normal vector to $\partial D$;
  \item[(2)] $F(x,a)=F(x,b)$ for all $x\in D$;
  \item[(3)] if $G(x,a)=0$ for some $x\in D$, then $G(x,b)=0$;
  \item[(4)] if $G(x,a)\neq0$, then
    \[
      G(x,a)G(x,b)<0.
    \]
\end{enumerate}
Then there exists $(\bar{x},\bar{s})\in D\times[a,b]$ such that
\[
  F(\bar{x},\bar{s})=0
  \qquad\text{and}\qquad
  G(\bar{x},\bar{s})=0.
\]
\end{theorem}

\begin{proof}
We first consider the case where $F$ is smooth and $0$ is a regular
value of $F$, $F_a=F|_{D\times\{a\}}$, and
$F_b=F|_{D\times\{b\}}$.

By condition~(1), the boundary map $F_a/\|F_a\|$ has degree $1$.
Hence, by the Poincar\'e--Hopf index theorem, $F_a$ has an odd number
of zero points in $D$.  Denote them by
\[
  x_1,\ldots,x_{2k+1}.
\]
By condition~(2), these are also precisely the zero points of $F_b$.

If $G(x_i,a)=0$ for some $i$, then $(x_i,a)$ is already a common
zero of $F$ and $G$.  We may therefore assume that
\[
  G(x_i,a)\neq0,\qquad i=1,\ldots,2k+1.
\]
Condition~(4) then shows that $G(x_i,a)$ and $G(x_i,b)$ have opposite
signs for every $i$.

Since $0$ is a regular value of $F$, the set $F^{-1}(0)$ is a compact
one-dimensional manifold.  By condition~(1), it does not meet
$\partial D\times[a,b]$, and hence its interval components have
endpoints among
\[
  (x_i,a),\ (x_i,b),\qquad i=1,\ldots,2k+1.
\]
For each $i$, exactly one of $(x_i,a)$ and $(x_i,b)$ is a positive
endpoint of $G$.  Thus the total number of positive endpoints is
$2k+1$, which is odd.

The endpoint-pairing idea underlying this argument already appeared
in~\cite{Yang26}. In that earlier proof, the zeros on the two boundary
faces were partitioned according to whether $G$ was positive or
negative, and the argument followed how the interval components of
$F^{-1}(0)$ pair the corresponding positive and negative endpoints.
The following parity argument gives a shorter and more symmetric
formulation of the same idea.

If the two endpoints of every interval component of $F^{-1}(0)$ had
the same sign under $G$, then the positive endpoints would occur in
pairs, and their total number would be even.  This is a contradiction.
Therefore some interval component $l\subset F^{-1}(0)$ has endpoints
at which $G$ has opposite signs.  By continuity, there exists
$(\bar{x},\bar{s})\in l$ such that
\[
  G(\bar{x},\bar{s})=0.
\]
Since $l\subset F^{-1}(0)$, we also have
\[
  F(\bar{x},\bar{s})=0.
\]

We now consider the continuous case.  Suppose, to the contrary, that
$F$ and $G$ have no common zero.  By compactness, there exists
$\delta>0$ such that
\[
  \|F(x,s)\|+|G(x,s)|\geq\delta
\]
for all $(x,s)\in D\times[a,b]$.

By a standard endpoint-preserving smooth approximation and
transversality argument, one may choose a smooth map
\[
  \overline{F}:D\times[a,b]\to\mathbb{R}^{2}
\]
such that
\[
  \overline{F}(x,a)=\overline{F}(x,b)
  \qquad\text{for all }x\in D,
\]
the value $0$ is regular for $\overline{F}$,
$\overline{F}_a$, and $\overline{F}_b$, and
\[
  \|F(x,s)-\overline{F}(x,s)\|<\frac{\delta}{4}
  \qquad\text{for all }(x,s)\in D\times[a,b].
\]
Taking the approximation sufficiently close, condition~(1) remains
valid for $\overline{F}$.

Applying the smooth case to $\overline{F}$ and $G$, we obtain
$(\bar{x},\bar{s})\in D\times[a,b]$ such that
\[
  \overline{F}(\bar{x},\bar{s})=0
  \qquad\text{and}\qquad
  G(\bar{x},\bar{s})=0.
\]
Consequently,
\[
  \|F(\bar{x},\bar{s})\|+|G(\bar{x},\bar{s})|
  <\frac{\delta}{4},
\]
contradicting the definition of $\delta$.
\end{proof}

\section{Table Problem}\label{sec:table}

In this section we first give a new proof of Fenn's theorem on the
square table problem. We then prove a variant under different
boundary conditions.

\begin{proof}[Proof of Theorem~\ref{thm:fenn}]
  Unlike the arguments in \cite{Fenn70}, we instead consider the
  space $D\times\bigl[0,\frac{\pi}{2}\bigr]$ and the maps
  \[
    G=\phi,\qquad F=\psi,
  \]
  where $\phi$ and $\psi$ are as defined in~\eqref{eq:Phi}, with the side 
  length fixed as $d$.  Note that $\psi(x,\alpha)=\psi(x,\alpha+\frac{\pi}{2})$ 
  by construction, so condition~(2) of Theorem~\ref{thm:main} holds.

  It is enough to prove that there is a point
  $(x,s)\in D\times\bigl[0,\frac{\pi}{2}\bigr]$ such that
  \[
    \phi(x,s)=0\qquad\text{and}\qquad\psi(x,s)=0.
  \]
  Indeed, on $\phi^{-1}(0)$ we have
  $\psi(x,s)=4v(x,s)$, and hence $\psi(x,s)=0$ if and only if
  $v(x,s)=0$.

  Suppose, to the contrary, that no such point exists. Let 
  $A=\{(x,s)\in D\times[0,\pi/2]:\phi(x,s)=0\}$. By compactness, 
  there exists $\delta>0$ such that $\|\psi(x,s)\|\geq \delta$ for all 
  $(x,s)\in A$.

  By the convexity of $D$, we have
\[
  \psi(x,s)\cdot\nu_{D}(x)\ge0
  \qquad\text{for every }x\in\partial D,
\]
where $\nu_{D}(x)$ is the outward unit normal to $\partial D$.
Indeed, the corresponding inequality holds for each of the four terms
$v(x,s+j\frac{\pi}{2})$, $j=0,1,2,3$, and hence also for their
sum $\psi(x,s)$.  Now it is easy to see that there exists a
smooth map
\[
  \overline\psi:
  D\times\bigl[0,\frac{\pi}{2}\bigr]\to\mathbb{R}^{2}
\]
such that
\[
  \overline\psi(x,s)\cdot\nu_{D}(x)>0
  \qquad
  \text{on }\partial D\times\bigl[0,\frac{\pi}{2}\bigr],
\]
\[
  \overline\psi(x,0)=\overline\psi\left(x,\frac{\pi}{2}\right),
\]
and
\[
  \|\overline\psi(x,s)-\psi(x,s)\|<\frac{\delta}{2}
  \qquad
  \text{for all }(x,s)\in
  D\times\bigl[0,\frac{\pi}{2}\bigr].
\]
Then by Theorem~\ref{thm:main} applied with $F=\overline\psi$ and
$G=\phi$, there exists a point $(\bar{x},\bar{s})\in A$ such that
$\overline\psi(\bar{x},\bar{s})=0$.  Therefore
\[
  \|\psi(\bar{x},\bar{s})\|<\frac{\delta}{2},
\]
leading to a contradiction.  Hence $A\cap\psi^{-1}(0)\neq\varnothing$,
and the square table is horizontally placed.
\end{proof}

The following new result concerns the boundary conditions.

\begin{theorem}\label{thm:newbc}
  Let $D\subset\mathbb R^2$ be a compact convex domain, and let
  $f:\mathbb R^2\to\mathbb R$ be a continuous map satisfying one of
  the following conditions:
  \begin{enumerate}
  	\item[(i)] $f(x)\geq 0$ for $x\in\mathbb R^2\setminus D$, and there
  	exists $\delta>0$ such that
  	\[
  	f(x)\leq 0
  	\]
  	whenever $\operatorname{dist}(x,\partial D)\leq\delta$;
  	\item[(ii)] $f(x)\leq 0$ for $x\in\mathbb R^2\setminus D$, and there
  	exists $\delta>0$ such that
  	\[
  	f(x)\geq 0
  	\]
  	whenever $\operatorname{dist}(x,\partial D)\leq\delta$.
  \end{enumerate}
  Then, for every $a<\frac{\sqrt2}{2}\delta$, there exists a point $p\in D$ and a square in $\mathbb R^2$ centered at $p$, with side length $a$, such that $f$ takes the same value at the four vertices of the square.
\end{theorem}

\begin{proof}
  It is enough to consider case~(i).  In this case, we consider the
  space $D\times\bigl[0,\frac{\pi}{2}\bigr]$ and the maps
  $\phi:D\times\bigl[0,\frac{\pi}{2}\bigr]\to\mathbb{R}$ and
  $\psi:D\times\bigl[0,\frac{\pi}{2}\bigr]\to\mathbb{R}^{2}$, where
  $\phi$ and $\psi$ are as defined in Section~\ref{sec:intro}, but with side
  length $a$.  Since the domain $D$ is convex, it is easy to prove
  that $-\psi(x,s)\cdot\nu_{D}(x)\ge0$ for each $x\in\partial D$, where
  $\nu_{D}(x)$ is the outward unit normal. As in the new proof of Fenn's 
  theorem, replacing $-\psi$ by a sufficiently close map satisfying 
  the strict boundary inequality, we can apply Theorem~\ref{thm:main} 
  with $F=-\psi$ and $G=\phi$ to complete the proof.
  Case~(ii) is proved in the same way, with the sign reversed.
\end{proof}

\section{Table Problem on Saddle Surfaces}\label{sec:surfaces}

In this section we first present a sufficient condition for a
square table to be placed horizontally on the surface defined by a
general graph $G(f)$.  For a continuous function
$f:\mathbb{R}^{2}\to\mathbb{R}$, in view of Theorem~\ref{thm:main} we have the following fact.

\begin{theorem}\label{thm:surface}
Let $D\subset\mathbb R^{2}$ be a compact domain with $C^{1}$ boundary
homeomorphic to $S^{1}$.  Let $\phi_r$ and $\psi_r$ be the maps defined
in~\eqref{eq:Phi}, where $r$ is the side length of the corresponding
square.  Assume that $\psi_r\neq0$ on
$\partial D\times[0,\pi/2]$, and that
\[
  \deg\left(
    \frac{\psi_r(\cdot,0)}{\|\psi_r(\cdot,0)\|}:
    \partial D\longrightarrow S^{1}
  \right)
\]
is odd.  Then one can place horizontally a square table with side
length $r$ on the graph $G(f)$, with its center at a point of $D$.
\end{theorem}

We are interested in table problems on some kinds of surfaces
defined by a special class of continuous functions
$f:\mathbb{R}^{2}\to\mathbb{R}$, the saddle surfaces as defined
below.

\begin{definition}\label{def:saddle}
  Consider a continuous map
  $f:\mathbb{R}^{2}\to\mathbb{R}$.  The graph of $f$ in
  $\mathbb{R}^{3}$ is called a \emph{saddle surface} with center
  $p=(p_{1},p_{2})$, if $f$ satisfies the following conditions:
  \begin{enumerate}
    \item[(i)] $f(x_{1},x_{2})>f(\bar{x}_{1},x_{2})$, provided
      $|x_{1}-p_{1}|>|\bar{x}_{1}-p_{1}|$;
    \item[(ii)] $f(x_{1},x_{2})<f(x_{1},\bar{x}_{2})$, provided
      $|x_{2}-p_{2}|>|\bar{x}_{2}-p_{2}|$.
  \end{enumerate}
\end{definition}

This definition is rather restrictive; in fact, as shown below, it
forces the function to be even in each coordinate with respect to the
center. Without loss of generality, we only consider saddle surfaces with
center at the origin.

Here are some examples of saddle surfaces.
\begin{example}\label{ex:saddle1}
  The surface defined by $x_{3}=f(x_{1},x_{2})=ax_{1}^{2}-bx_{2}^{2}$,
where $a,b>0$.
\end{example}

\begin{example}\label{ex:saddle2}
  The surface defined by $x_{3}=f(x_{1},x_{2})=g(x_{1})-h(x_{2})$,
  where $g$ and $h$ are even functions and are strictly increasing on
$[0,\infty)$.
\end{example}

We are interested in the following questions.

\noindent\textbf{Q1.} Can we balance horizontally an equilateral
triangular table on the surface defined by
$x_{3}=f(x_{1},x_{2})=ax_{1}^{2}-bx_{2}^{2}$?

\noindent\textbf{Q2.} Can we balance horizontally a rectangular
table on the surface defined by
$x_{3}=f(x_{1},x_{2})=g(x_{1})-h(x_{2})$ as given in
Example~\ref{ex:saddle2}?

More generally, can we balance horizontally an equilateral triangle
table or a rectangular table on a saddle surface defined in
Definition~\ref{def:saddle}?

It is easy to see that a square table of any size can be placed
horizontally on the saddle surface in Example~\ref{ex:saddle2}.  
In fact, the same conclusion holds for every saddle surface in 
the sense of Definition~\ref{def:saddle}.  This statement was 
conjectured in~\cite{Yang26}, and we now prove it.

\begin{theorem}[Horizontal squares on saddle surfaces]\label{thm:saddle}
  Let $f:\mathbb{R}^{2}\to\mathbb{R}$ be a continuous map such that
  $G(f)$ is a saddle surface with its center at the origin.  Then
  for any square table $S$, one can place $S$ horizontally on the
  surface $G(f)$, with its center at the origin and with its sides
  parallel to the coordinate axes.
\end{theorem}

\begin{proof}
  We first show that $f$ is even in each coordinate.  Fix
  $x_2\in\mathbb R$ and $t\ge0$.  For every $\varepsilon>0$, the
  defining condition in the $x_1$-direction gives
  \[
    f(t+\varepsilon,x_2)>f(-t,x_2),
    \qquad
    f(-t-\varepsilon,x_2)>f(t,x_2).
  \]
  Letting $\varepsilon\to0$ and using the continuity of $f$, we obtain
  \[
    f(t,x_2)\ge f(-t,x_2),
    \qquad
    f(-t,x_2)\ge f(t,x_2).
  \]
  Hence
  \[
    f(t,x_2)=f(-t,x_2).
  \]
  Since this holds for every $t\ge0$, $f$ is even in the first
  coordinate.  The same argument, using the defining condition in the
  $x_2$-direction, shows that $f$ is even in the second coordinate.

  Let the side length of $S$ be $r$.  Place its projection on the
  $x_1x_2$-plane as the square with vertices
  \[
    \left(\frac r2,\frac r2\right),\quad
    \left(\frac r2,-\frac r2\right),\quad
    \left(-\frac r2,\frac r2\right),\quad
    \left(-\frac r2,-\frac r2\right).
  \]
  By the evenness of $f$ in both coordinates, the four values of $f$
  at these vertices are equal.  Hence the corresponding four points on
  $G(f)$ lie in a horizontal plane.  This gives the desired horizontal
  placement.
\end{proof}

\section{Balancing Horizontally a Cyclic Table}\label{sec:cyclic}

For the convenience of the arguments in the sequel, we give the
following definition.  By a \emph{cyclic} configuration of four points,
we mean a four-point set $V=[v_1,v_2,v_3,v_4]\subset\mathbb R^2$
which is contained in a round circle.

\begin{definition}
  Given a continuous map $f:\mathbb{R}^{2}\to\mathbb{R}$, the graph
  $G(f)$ is called a \emph{Fenn graph over $D$} if $f$ satisfies the
  conditions in Fenn's theorem for the compact convex domain
  $D\subset\mathbb{R}^{2}$.
\end{definition}

A typical configuration besides the square is the rectangular
configuration.  Up to now, no solid theory for balancing
horizontally a rectangle on a Fenn graph has appeared in the
literature.  A well-known fact is that any rectangle can be
balanced on every continuous graph $G(f)$ (which means that the
table does not wobble) due to a topological theorem by
Livesay~\cite{Livesay54}, which stated that for every
continuous function $F:S^{2}\to\mathbb{R}$ and every angle
$\theta\in(0,\pi)$, there exist two diameters of $S^{2}$ making
angle $\theta$ such that $F$ takes the same value at the four
endpoints of these two diameters.  Since the vertices of a
rectangle inscribed in a circle are exactly the endpoints of two
diameters making a fixed angle, Livesay's theorem implies that
any prescribed rectangle can be positioned on an arbitrary
continuous graph so that its four legs touch the graph.  In other
words, the table can be balanced in the sense that it does not
wobble.  However, this does not mean that the tabletop is
horizontal, since the common plane determined by the four contact
points need not be parallel to the $x_{1}x_{2}$-plane.

The conjecture given in \cite{Yang26} concerns the statement that like the
case of square table, and rectangle table can also be placed horizontally
on the Fenn's graphs, but no prove can be found in the literature. Nonetheless, 
there is evidence suggesting that one can place horizontally a
rectangular table with any aspect ratio but of a restricted size on
a smooth Fenn graph, and this is even possible for any cyclic
configuration $V=[v_{1},v_{2},v_{3},v_{4}]\subset\mathbb{R}^{2}$
with its size (the diameter of $V$) being restricted.

The following result is not a fixed-size rectangular table theorem.
Rather, it shows that on almost every level set one can find a
rectangle of any prescribed aspect ratio, up to similarity.

\begin{theorem}[Rectangles on regular levels of a Fenn graph]\label{thm:rect}
  Let $f:\mathbb{R}^{2}\to\mathbb{R}$ be a smooth map such that
  $G(f)$ is a Fenn graph over a compact convex domain
  $D\subset\mathbb{R}^{2}$.  Let $h=\max_{D}f$.  Then there exists a
  subset $K\subset(0,h)$ of full Lebesgue measure such that for
  every $c\in K$ and every rectangle
  $R=[v_{1},v_{2},v_{3},v_{4}]\subset\mathbb{R}^{2}$, there is an
  orientation-preserving similarity
  $T:\mathbb{R}^{2}\to\mathbb{R}^{2}$ such that
  \[
    f(Tv_{1})=f(Tv_{2})=f(Tv_{3})=f(Tv_{4})=c.
  \]
  Equivalently, the four points
  \[
    (Tv_{i},f(Tv_{i}))\in G(f),\qquad i=1,2,3,4,
  \]
  lie in a horizontal plane and form a rectangle similar to $R$. 
  If $R$ is centered at the origin, then $T$ may be chosen so that $T(0)\in D$.
\end{theorem}

\begin{proof}
  Let $K$ be the set of regular values of $f$ in $(0,h)$.  By
  Sard's theorem, $K$ has full Lebesgue measure in $(0,h)$.

  Fix $c\in K$.  Since $c$ is a regular value of $f$, the level set
  $f^{-1}(c)$ is a smooth one-dimensional submanifold of
  $\mathbb{R}^{2}$.  Moreover, since $G(f)$ is a Fenn graph, we
  have $f=0$ outside $D$ and $f\ge0$ on $D$.  Thus, for $0<c<h$,
  the set $f^{-1}(c)$ is contained in the interior of $D$.  In
  particular, $f^{-1}(c)$ is compact.

  Therefore every connected component of $f^{-1}(c)$ is a smooth
  Jordan curve.  Choose one component and denote it by $\gamma$.
  By the rectangular peg theorem of Greene and
  Lobb~\cite{GreeneLobb21}, every smooth Jordan curve in the plane
  inscribes a rectangle similar to any prescribed rectangle.  Hence
  there exists an orientation-preserving similarity
  $T:\mathbb{R}^{2}\to\mathbb{R}^{2}$ such that
  \[
    Tv_{1},Tv_{2},Tv_{3},Tv_{4}\in\gamma\subset f^{-1}(c).
  \]
  It follows that
  \[
    f(Tv_{1})=f(Tv_{2})=f(Tv_{3})=f(Tv_{4})=c.
  \]
  Thus the corresponding four points on $G(f)$ all have height $c$,
  and so they lie in a horizontal plane.  
  
  If $R$ is centered at the origin, then
\[
0=\frac{1}{4}(v_1+v_2+v_3+v_4).
\]
Since $T$ is affine,
\[
T(0)=\frac{1}{4}\sum_{i=1}^{4}T(v_i)\in D,
\]
because $T(v_i)\in\gamma\subset D$ and $D$ is convex.
  This completes the proof.
\end{proof}

Similarly, by applying the theorem of Greene and Lobb on cyclic
quadrilaterals and smooth Jordan curves~\cite{GreeneLobb23}, we obtain
the following result for cyclic quadrilaterals.

\begin{theorem}[Cyclic quadrilaterals on regular levels of a Fenn graph]
\label{thm:cyclic-regular-levels}
Let $f:\mathbb{R}^{2}\to\mathbb{R}$ be a smooth map such that
$G(f)$ is a Fenn graph over a compact convex domain
$D\subset\mathbb{R}^{2}$.  Let $h=\max_{D}f$. Then there exists a subset $K\subset(0,h)$ of full Lebesgue measure
such that, for every $c\in K$ and every cyclic quadrilateral $Q=(v_{1},v_{2},v_{3},v_{4})$
 in the Euclidean plane, there exists an orientation-preserving  similarity $T:\mathbb{R}^{2}\longrightarrow\mathbb{R}^{2}$
such that
\[
  f(Tv_{1})=f(Tv_{2})=f(Tv_{3})=f(Tv_{4})=c.
\]
Equivalently, the four points
\[
  (Tv_i,f(Tv_i))\in G(f),\qquad i=1,2,3,4,
\]
lie in a horizontal plane and form a cyclic quadrilateral similar
to $Q$.

Moreover, if $Q$ is inscribed in a circle centered at the origin and
\[
  0\in\operatorname{conv}\{v_{1},v_{2},v_{3},v_{4}\},
\]
then the similarity $T$ may be chosen so that
\[
  T(0)\in D.
\]
\end{theorem}

In view of the above theorems, it is natural to state the
following conjecture.

\begin{conjecture}\label{conj:cyclic}
  Let $G(f)$ be a Fenn graph over a compact convex domain
  $D\subset\mathbb{R}^{2}$.  Then for every rectangle, and 
  more generally for every cyclic quadrilateral
  $Q=(v_{1},v_{2},v_{3},v_{4})$ inscribed in the circle
  $S^{1}(r)=\{x\in\mathbb{R}^{2}: \|x\|=r\}$, there exists a
  rigid motion $T:\mathbb{R}^{2}\to\mathbb{R}^{2}$ such that
  \[
    f(Tv_{1}) = f(Tv_{2}) = f(Tv_{3}) = f(Tv_{4})
    \quad\text{and}\quad T(0)\in D.
  \]
\end{conjecture}

\begin{ack}
This work was supported by the National Natural Science Foundation of China
(Grant No. 12531017) and the Hubei Provincial Natural Science Foundation of
China (Grant No. 2026AFA034).
\end{ack}

\bibliographystyle{amsplain}
\bibliography{newtable}

@article{Fenn70,
  author  = {Roger Fenn},
  title   = {The Table Theorem},
  journal = {Bull. London Math. Soc.},
  volume  = {2},
  year    = {1970},
  pages   = {73--76},
  mrclass = {57.00},
  mrnumber= {0271940}
}

@article{Livesay54,
  author  = {G. R. Livesay},
  title   = {On a Theorem of {F.\ J.\ Dyson}},
  journal = {Ann. of Math. (2)},
  volume  = {59},
  year    = {1954},
  pages   = {227--229}
}

@article{GreeneLobb21,
  author  = {Joshua Evan Greene and Andrew Lobb},
  title   = {The Rectangular Peg Problem},
  journal = {Ann. of Math. (2)},
  volume  = {194},
  number  = {2},
  year    = {2021},
  pages   = {509--517}
}

@article{GreeneLobb23,
  author  = {Joshua Evan Greene and Andrew {Lobb}},
  title   = {Cyclic Quadrilaterals and Smooth {J}ordan Curves},
  journal = {Invent. Math.},
  volume  = {234},
  year    = {2023},
  pages   = {931--935}
}

@article{Yang26,
  author  = {Xiao-Song Yang},
  title   = {Table Problem Revisited},
  journal = {arXiv:2606.23089 [math.GT]},
  year    = {2026},
  note    = {\url{https://arxiv.org/abs/2606.23089}}
}

\end{document}